\documentclass[12 pt]{amsart}
\usepackage{srcltx}
\usepackage{eurosym}
\usepackage{mathtools}
\usepackage{amsmath}
\usepackage{amsfonts}
\usepackage{amssymb}
\usepackage{amsthm}
\usepackage{graphicx}
\usepackage{mathrsfs}
\usepackage{xcolor}
\usepackage{exscale}
\usepackage{latexsym}
\usepackage{dirtytalk}
\usepackage{graphicx}
\usepackage[normalem]{ulem}
\usepackage{fancyhdr}

\usepackage[margin=3cm]{geometry}
\usepackage{hyperref}
\hypersetup{
 colorlinks   = true,
 urlcolor     = blue,
 linkcolor    = blue,
 citecolor   = red ,
 bookmarksopen=true
}

\def\YYint#1#2#3{{\setbox0=\hbox{$#1{#2#3}{\iint}$}
    \vcenter{\hbox{$#2#3$}}\kern-.50\wd0}}
\def\Xint#1{\mathchoice
	{\XXint\displaystyle\textstyle{#1}}%
	{\XXint\textstyle\scriptstyle{#1}}%
	{\XXint\scriptstyle\scriptscriptstyle{#1}}%
	{\XXint\scriptscriptstyle\scriptscriptstyle{#1}}%
	\!\int}
\def\XXint#1#2#3{{\setbox0=\hbox{$#1{#2#3}{\int}$}
		\vcenter{\hbox{$#2#3$}}\kern-.5\wd0}}
\def\aver#1{\Xint-_{#1}}

\def\Xint#1{\mathchoice
    {\XXint\displaystyle\textstyle{#1}}%
    {\XXint\textstyle\scriptstyle{#1}}%
    {\XXint\scriptstyle\scriptscriptstyle{#1}}%
    {\XXint\scriptscriptstyle\scriptscriptstyle{#1}}%
      \!\int}
\def\XXint#1#2#3{{\setbox0=\hbox{$#1{#2#3}{\int}$}
    \vcenter{\hbox{$#2#3$}}\kern-.50\wd0}}

\makeatletter
\def\namedlabel#1#2{\begingroup
   \def\@currentlabel{#2}%
   \label{#1}\endgroup
}
\makeatother
\makeatletter
\newcommand{\rmh}[1]{\mathpalette{\raisem@th{#1}}}
\newcommand{\raisem@th}[3]{\hspace*{-1pt}\raisebox{#1}{$#2#3$}}
\makeatother

\newtheorem{thm}{Theorem}[section]

\newtheorem{lem}[thm]{Lemma}

\theoremstyle{definition}
\newtheorem{defn}[thm]{Definition}
\theoremstyle{remark}
\newtheorem{rem}[thm]{Remark}
\theoremstyle{example}

\numberwithin{equation}{section}
\newcommand{\R}{\mathbb{R}}

\newcommand{\ve}{\varepsilon}

\begin{document}
\title[]
 {Borderline gradient continuity for the normalized $p$-parabolic  operator}

\author{Murat Akman}
\address{University of Essex, Colchester CO4 3SQ,
United Kingdom}
\email[Murat Akman]{murat.akman@essex.ac.uk}

 \author{Agnid Banerjee}
\address{Tata Institute of Fundamental Research\\
Centre For Applicable Mathematics \\ Bangalore-560065, India}\email[Agnid Banerjee]{agnidban@gmail.com}

\author{Isidro H. Munive }
\address{CUCEI, University of Guadalajara}\email[Isidro H. Munive]{isidro.munive@academicos.udg.mx}

\thanks{A.B was supported in part by  Department of Atomic Energy,  Government of India, under
project no.  12-R \& D-TFR-5.01-0520.}

% \thanks{First author is supported in part by SERB Matrix grant MTR/2018/000267}
 %\thanks{Second author is supported by CONACYT grant 265667, Instituto de Matem\'aticas, UNAM}
%\author[,]
%{}
%\address{R. B. Verma \hfill\break
% TIFR \newline
% .}
% \email{...}
%\address{Agnid Banerjee\hfill\break
% TIFR \newline
% ....}
%\email{...}

\subjclass[2010]{Primary 35K55, 35B65, 35D40.}

\begin{abstract}
In this paper, we prove  gradient continuity estimates for viscosity  solutions to  $\Delta_{p}^N u - u_t= f$ in terms of the  scaling critical $L(n+2,1 )$  norm of $f$,   where $\Delta_{p}^N$ is the game theoretic  normalized $p-$Laplacian operator defined in \eqref{pl} below.  Our main result, Theorem \ref{main} constitutes  borderline gradient  continuity estimate for $u$ in terms of the modified parabolic Riesz potential $\mathbf{P}^{f}_{n+1}$ as defined  in \eqref{Riesz} below. Moreover, for $f \in L^{m}$ with $m>n+2$, we also obtain H\"older continuity of the spatial gradient of the solution $u$,  see Theorem \ref{main1} below. This  improves  the gradient H\"older continuity result  in  \cite{AP} which considers bounded $f$.   Our main results Theorem \ref{main} and Theorem \ref{main1} are  parabolic analogues of  those in \cite{BM}. Moreover  differently from that in \cite{AP},  our  approach  is independent of the  Ishii-Lions method which is crucially used in \cite{AP}   to obtain Lipschitz estimates for   homogeneous perturbed equations as an intermediate step. \end{abstract}
\maketitle

\tableofcontents 

\section{Introduction}
The purpose of this paper is to   obtain pointwise gradient continuity estimates   for viscosity solutions to
\begin{equation}
\label{m}
\Delta_{p}^N u-u_t= f\quad \text{in $Q_1\overset{def}= B_1\times (-1,0]$},\ 1<p<\infty,
\end{equation}
in terms of the scaling critical $L(n+2,1)-$norm of $f$. Here, $\Delta_{p}^N $ denotes the  game theoretic normalized $p-$Laplace operator given by
\begin{equation}
\label{pl}
\Delta_{p}^N u\overset{def}= \bigg(\delta_{ij}+ (p-2) \frac{u_i u_j}{ |\nabla u|^2} \bigg) u_{ij},
\end{equation}
that arises in tug of war games with noise (see \cite{MPR}) and also image processing (see \cite{Do}).

It  is well known that   borderline or end-point regularity estimates play a fundamental role in the theory of elliptic and parabolic partial differential equations.  In order to put our  main result in the right perspective, we note that in 1981,  E. Stein in his  work \cite{MR607898} showed the following.
\begin{thm}\label{stein}
Let $L(n,1)$ denote the standard Lorentz space, then the following implication holds:
\[\nabla v \in L(n,1) \ \implies \ v\  \text{\emph{is continuous}}.\]  
\end{thm}
The Lorentz space $L(n,1)$ appearing in Theorem \ref{stein}  consists of those measurable functions $g$ satisfying the condition
\[
\int_{0}^{\infty} |\{x: g(x) > t\}|^{1/n} dt < \infty.
\]
Theorem \ref{stein} can be regarded as the end point case of the Sobolev-Morrey embedding.   Moreover,  Theorem \ref{stein} coupled with the standard Calderon-Zygmund theory  has the following interesting consequence.
\begin{thm}\label{st}
$\Delta u \in L(n,1) \implies \nabla u$ is continuous.
\end{thm} 
A similar result holds  in the parabolic situation for more general variable coefficient operators   when $f \in L(n+2, 1)$. As is well known, $"n+2"$ is the right parabolic dimension which   is dictated by the one parameter family of  parabolic scalings $\{\delta_r\}$ given by $\delta_r(x,t)= (rx, r^2 t)$ that keeps the parabolic structure invariant.  The analogue of Theorem \ref{st}  for general nonlinear and possibly degenerate elliptic and parabolic equations has required the development of a  rather sophisticated and powerful nonlinear potential theoretic methods (see for instance \cite{MR2823872,MR2900466,MR3174278} and the references therein).  The first breakthrough in this direction came up in the work  of Kuusi and Mingione in \cite{MR3004772} where they showed that the analogue of Theorem \ref{st}  holds for operators modelled after the variational $p$-Laplacian. Such a result  was subsequently generalized to $p$-Laplacian type systems by the same authors in  \cite{MR3247381}.   Such results were further extended to  degenerate parabolic equations and systems modelled on

\begin{equation}\label{degp}
\operatorname{div}(|\nabla u|^{p-2} \nabla u) - u_t =f,\ f \in L(n+2, 1)
\end{equation}
in \cite{KM1, KM2, KM3}.
 For  other local and global borderline gradient continuity results for various kinds of nonlinear equations, we refer to \cite{AB, CMa1, CMa2,  DKM, DP, K}.  Very recently,  the two of us in \cite{BM}  obtained a similar  borderline  gradient continuity result for the following inhomogeneous normalized $p$-Poisson equation 
 \[
 \Delta_p^N u = f,\ p>1, f \in L(n,1),\]
 by arguments based on non-divergence form techniques and compactness arguments which have their roots in the fundamental work of Caffarelli in \cite{Ca}.  The results in \cite{BM} sharpens the previous results obtained in \cite{APR} that dealt with $f$ belonging to subcritical function spaces.

The purpose of this work is to  obtain analogous gradient regularity results for solutions to the parabolic normalized $p-$Poisson problem  \eqref{m}. Our main results Theorem \ref{main} and Theorem \ref{main1} are parabolic counterparts of those in \cite{BM}.

We now  mention that over the last decade, there has been a growing  attention on equations of the type
\begin{equation}\label{nplap}
\Delta_p^N u-u_t=0
\end{equation} 
because of  their connections to tug-of-war games with noise. This aspect was first studied in \cite{PS} in the stationary case for the infinity Laplacian.  In recent times,   the parabolic  normalized $p-$Laplacian,  as well as its degenerate and singular variants, have been studied in a  variety of  contexts in  several papers, see \cite{A, JK, Do, BG1, BG2, BG3, HL,  PR, Ju, MPR}. See also the recent survey article \cite{P} for a more comprehensive account.  As previously mentioned,   such equations have also found applications in image processing (see for instance \cite{Do}). The  gradient H\"older continuity result for solutions to \eqref{nplap} was established in \cite{JS}.  See also \cite{IJS} for further interesting generalizations. The regularity result in \cite{JS} was subsequently extended to equations of the type
\[
\Delta_p^N u- u_t= f, \ f \in L^{\infty}
\]
in \cite{AP}.  Our main results thus further refine the result in \cite{AP} by allowing  $f$ to belong to more general (and possibly scaling critical) function space.

We would also like to mention that although most of the published works related to the equation \eqref{nplap} appeared only in recent years, we
have seen an unpublished handwritten note by N. Garofalo from 1993 where this equation along with its inhomogeneous variants were introduced.  As a matter of fact, in that note, there is a computation which leads to Lemma 3.1 in \cite{JS} that is regarding a subsolution type property ( or equivalently a quantitative maximum principle)  for the gradient.   This
is, up to our knowledge, the first time when \eqref{nplap} was studied and where  it was recognized that such an equation should have
good regularization properties.

Now a few words regarding our approach:  We  note that in general,  getting $C^{1}-$regularity result  amounts to show that the graph of $u$ can be touched by an affine function so that the error is of order $o(r)$ in a parabolic cylinder  of size $r$ for every $r$ small enough. The proof of this is based on iterative argument where one ensures improvement of flatness at  every successive scale by comparing to a solution of a  limiting equation with  more regularity.  At each step, via rescaling, it reduces to showing that if $\langle p_0,x\rangle  +u(x,t)$ solves \eqref{m} in $Q_1$ (see Section \ref{n} for relevant notations), then the oscillation of $u$ is strictly smaller in a smaller cylinder  "modulo" a linear function. This is accomplished via compactness arguments which crucially relies on apriori estimates. Such estimates in the context of $\Delta_{p}^N - \partial_t$ come from the Krylov-Safonov theory because the equation \eqref{m} lends itself a uniformly parabolic structure.  

\medskip

Now, for a $u$ that solves \eqref{m}, we have that $v=u-\langle p_0,x\rangle$ is a solution of the following perturbed equation
\begin{equation}\label{it}
\bigg(\delta_{ij} + (p-2) \frac{(v_i +(p_0)_i)(v_j+(p_0)_j)}{|\nabla v + p_0|^2} \bigg) v_{ij}- v_t=f.
\end{equation}
Therefore, in order to obtain improvement of flatness at each scale after a rescaling, it is necessary  to get uniform $C^{1}-$type estimates independent of $|p_0|$  for  the limiting equations corresponding to the case  $f \equiv 0$.  This is precisely what has been  done in \cite{AP} (see also \cite{APR})   by an adaptation of  the Ishii-Lions approach in \cite{IL},  using which  the authors   obtained uniform Lipschitz estimates  for solutions to \eqref{it} for large $|p_0|'s$ when $f=0$.  In this paper, we follow an approach which is different from that in \cite{AP,APR} and is inspired by ideas in the stationary case developed by the two of us  in \cite{BM}.  Our proofs of Theorem \ref{main} and Theorem \ref{main1}  are based   on separation of  the degenerate and the non-degenerate phase,  and  do not rely on the uniform Lipschitz estimates for equations of the type \eqref{it}.  This is inspired by ideas in \cite{W}, where an alternate proof of $C^{1, \alpha}-$regularity for the $p-$Laplace equation was given (see also \cite{Db, Le, To} for the first results on $C^{1,\alpha}$ type regularity for $p$-Laplace type equations).  Moreover, in the case when  $f$ is bounded, compared to that in \cite{AP},  our method  also provides a different  proof of the $C^{1, \alpha}-$ type regularity result for \eqref{m}.  Finally we mention that although our proofs follow some key ideas for the stationary case as  previously used in  \cite{BM}, nevertheless it entails some delicate adaptations  which are intrinsic to the parabolic situation. For instance, the proofs of the main approximation lemmas, i.e.  Lemma \ref{app1} and Lemma \ref{rt} involve compactness arguments that crucially rely   on stability results which are established in the course of this work.  It is to be mentioned that the  classical stability results for viscosity solutions as in \cite{CCS} do not apply in our situation because of the non-smooth dependence of the normalized $p$-Laplace operator on the \say{gradient} variable.  An alternate characterization of viscosity solutions to the homogeneous equation \eqref{nplap} as in Lemma \ref{equiv1} below plays a pervasive role in the proof of such stability arguments.   In closing,  we refer to \cite{AS, HL, LY} for similar regularity results for other variants of the normalized $p-$parabolic operators.  

The paper is organized as follows. In Section \ref{n}, we introduce various notations, notions and gather some preliminary results that are relevant to the present work and then state our main results. In Section \ref{mn}, we prove our main results. 

\section{Notations, Preliminaries and  statement of the main results}\label{n}
A generic point in space time $\R^n \times \R$ will be denoted by $(x,t), (y,s)$ etc. We denote by $B_r(x)$, the Euclidean ball of radius $r$ centered at $x$.  When $x=0$, we will denote such a set by $B_r$. By $\partial B_r(x)$, we will denote the boundary of the set  $B_{r}(x)$.  For $(x,t)\in\R^{n+1}$, we let
\[
Q_r(x,t)=Q_r+(x,t),\quad \text{where} \quad Q_r=B_r\times (-r^2,0].
\]

The parabolic boundary of $Q_r(x,t)$ will be denoted by $\partial_p Q_r(x,t)$.  The distance between two points in space time is defined as
\begin{equation}\label{dt1}
|(x, t) - (y,s)| \overset{def}= |x-y| + |t-s|^{1/2}.
\end{equation}

For notational ease $\nabla u$ will refer to the quantity  $\nabla_x u$.  The partial derivative in $t$ will be denoted by $\partial_t u$ and also at times  by $u_t$. The partial derivative $\partial_{x_i} u$  will be denoted by $u_i$.

Given $0< \lambda< \Lambda$ and 
$f \in L^{q}$,   $\mathcal{S}(\lambda, \Lambda, f)$ will denote  the set of all functions $u$ which solves  the following differential inequalities in the $W^{2,q}$ viscosity sense 
 \begin{equation}\label{dif1}
 u_t-\mathcal{P}_{\lambda, \Lambda}^{+} (\nabla^2 u) \leq f \leq  u_t-\mathcal{P}_{\lambda, \Lambda}^{-}  (\nabla^2 u).  
 \end{equation}
 We refer to \cite{CCS} for the precise notion of $W^{2,1,q}$ viscosity solutions. The operators $\mathcal{P}_{\lambda, \Lambda}^{-}$ and $ \mathcal{P}_{\lambda, \Lambda}^{+}$ appearing in \eqref{dif1} are  the minimal and maximal Pucci operators, respectively, defined in the following way
 \begin{equation}\label{max}
 \begin{cases}
 \mathcal{P}_{\lambda, \Lambda}^{-}( M)= \text{inf}_{\{A \in S(n): \lambda \mathbb{I} \leq A \leq \Lambda \mathbb{I}\}} \text{trace}\ (AM), 
 \\
 \mathcal{P}_{\lambda, \Lambda}^{+}( M)= \text{sup}_{\{A \in S(n): \lambda \mathbb{I} \leq A \leq \Lambda \mathbb{I}\}} \text{trace}\ (AM),
 \end{cases}
 \end{equation} 
where $S(n)$ the space of $n \times n$ symmetric matrices.    

By $C^{2,1}_{loc}$, we refer to the class of functions $\phi$ such that $\nabla \phi, \nabla^2 \phi$ and $\phi_t$ exists classically and  are   locally continuous. Likewise, $W^{2,1,q}_{\mbox{\tiny loc}}$ denotes the class of functions $\phi$ such that the distributional  $\nabla \phi, \nabla^2 \phi$ and $\phi_t$ are in $L^q_{\mbox{\tiny loc}}$.

We now turn our attention to the relevant notion of solution to \eqref{m}. For $p \in \R^n -\{0\}$ and $X=[m_{ij}] \in  S(n)$, following \cite{BM},  let
\[
F(q, X)= \bigg(\delta_{ij} + (p-2) \frac{q_i q_j}{|q|^2} \bigg)  m_{ij}.
\]
Then, as in \cite{CIL},  the lower semicontinuous relaxation $F_*$ is defined as follows
\begin{equation}
F_*(q,X)=\begin{cases}  \qquad F(q,X)\quad&\hbox{if }q\not=0,\\
\inf_{a\in{\mathbb R}^n\setminus\{ 0\}}F(a,X)\quad&\hbox{if }q=0,      
\end{cases}\end{equation}
while the upper semicontinuous relaxation $F^{*}$ is defined as 

\begin{equation}\label{up}
F^{*}(q,X)=\begin{cases} \qquad F(q,X)\quad&\hbox{if }q\not=0,\\
\sup_{a\in{\mathbb R}^n\setminus\{ 0\}}F(a,X)\quad&\hbox{if }q=0.     
\end{cases}
\end{equation}
\begin{defn}\label{defvs}
We say that $u$ is a $W^{2,1,q}$ viscosity sub-solution of \eqref{m} in a domain $Q=\Omega\times (0,T]$ in space-time,  if given  $\phi \in W^{2,1,q}$ such that $u-\phi$ has a local maximum at  $(x_0, t_0) \in Q$, then one has
\begin{equation}\label{test}
\limsup_{(x,t) \to (x_0,t_0)}\left( F^{*} ( \nabla \phi(x,t), \nabla^2 \phi(x,t))-\varphi_t - f(x,t)\right) \geq 0.
\end{equation}
\end{defn}
In an analogous way, the notion of viscosity supersolution of \eqref{m} is defined  using $F_*$ instead of $F^{*}$, and where $\limsup$ gets replaced by $\liminf$ in the equation \eqref{test} above.  Finally, we  say that $u$ is a $W^{2,1,q}$ viscosity solution to \eqref{m} if it is both a subsolution and a supersolution.  It is easy to  deduce that if $u$ is a $W^{2,1,q}$ viscosity solution to \eqref{m}, then $u$ belongs to the Pucci class $\mathcal{S}(\lambda, \Lambda, f)$ in the $W^{2,1,q}$ viscosity sense where  
\begin{equation}\label{l}
\lambda=\text{min}(1, p-1)\quad \mbox{and}\quad  \Lambda= \text{max}(1, p-1).
\end{equation}

\medskip

Now in the case when $f=0$, it follows from  Lemma 2.1 in \cite{MPR} (see also Proposition 2.8 in \cite{BG2})  that the following equivalent characterisation of viscosity solutions hold. Such an equivalent characterization will  play a crucial role in our analysis.

\begin{lem}\label{equiv1}
$u$ is a $C^{2,1}$ viscosity solution to 
\[
\Delta_p^N u - u_t =0
\]
in $\Omega_T \overset{def} =\Omega \times (0, T]$ if and only if whenever $(x_0,t_0) \in \Omega_T$ and  $\phi \in C^{2,1}(\overline{\Omega_T})$ is such that:
\begin{enumerate}
\item[(a)] $u(x_0, t_0) = \phi(x_0, t_0)$,
\item[(b)] $ u(x, t) > \phi(x, t)$  for $(x, t) \in \Omega_T$ , $(x, t) \neq (x_0, t_0)$,
\end{enumerate}
then at the point $(x_0, t_0)$, we have:
\begin{enumerate}
\item[(i)] $\phi_t \geq \Delta_p^N \phi$ if $\nabla \phi(x_0, t_0) \neq 0$,
\item[(ii)] $\phi_t(x_0, t_0)\geq 0$ if $\nabla \phi(x_0, t_0)=0$, and $\nabla^2\phi(x_0, t_0)=0.$
\end{enumerate}
Moreover, we require that when testing from above, all the inequalities are reversed.
\end{lem}
Lemma \ref{equiv1} can be thought of as the parabolic analogue of the following interesting result in \cite{JLM} which states that for $p-$Laplace equation, i.e.
\[
\operatorname{div}(|\nabla u|^{p-2} \nabla u) =0,
\]
 the notion of viscosity solution  is only required to be tested at the points where the gradient of the test function does not vanish. This was crucially used in the proof of the corresponding gradient continuity result  in \cite{BM} for the time independent case. In an entirely analogous way as in \cite{BM}, Lemma \ref{equiv1} plays a critical role in the proof of the corresponding stability result in Lemma \ref{rt} below which is an important ingredient in the "phase separation" argument in the proof of Theorem \ref{main} and Theorem \ref{main1}.  Note that the stability result in \cite{CCS} does not apply in our situation because of the singular dependence of the normalized $p$-Laplacian operator on the gradient variable.

Finally, we mention that the  Lorentz space $L(n+2,1)$  consists of those measurable functions $g:\R^{n+1}\rightarrow\R$ satisfying the condition
\[
\int_{0}^{\infty} |\{(x,t): g(x,t) > s\}|^{\frac{1}{n+2}} ds < \infty.
\]

\begin{rem}\label{mod}
 For a function $f: \R^{n+1} \to \R$, we define the  modified Riesz   potential $\mathbf{P}^f_q(x_0, t_0, r)$ as
%\begin{equation}
%\label{Riesz}
%\mathbf{P}^f_q(x_0,t_0,r)\overset{def}= \int^r_0\left(\def\avint{\mathop{\,\rlap{-}\!\!\int}\nolimits}\avint_{Q_{\rho}(x_0,t_0)}|f|^qdxdt\right)^{\frac{1}{q}}d\rho,
%\end{equation} 
\begin{equation}
\label{Riesz}
\mathbf{P}^f_q(x_0,t_0,r)\overset{def}= \int^r_0\left(\aver{Q_{\rho}(x_0,t_0)}|f|^qdxdt\right)^{\frac{1}{q}}d\rho,
\end{equation} 
where  $Q_{\rho}(x_0, t_0) = B_{\rho}(x_0) \times (t_0-\rho^2, t_0]$ denotes the standard parabolic cylinder with vertex at $(x_0,t_0)$ and width $\rho > 0$.

From Lemma 2.3 in \cite{KM1} we get that for every $(x_0,t_0)\in \R^{n+1}$
\begin{align}\label{int1}
 \mathbf{P}^f_q(x_0,t_0,r)\leq  c_1 \int_{0}^{\omega_n r^{n+2}} \left[ (|f|^q)^{**}(\rho) \rho^{\frac{q}{n+2}} \right]^{\frac{1}{q}} \ \frac{d\rho}{\rho},
\end{align}
where the constant $c$ depends only on $n$, $\omega_n$ denotes the measure of the unit-ball, and  $(|f|^q)^{**}$ is defined as 
\[
(|f|^q)^{**}(\rho)= \frac{1}{\rho} \int_{0}^{\rho}  (|f|^q)^{*}(s) ds,
\]
with $(|f|^q)^{*}(\rho)$ being the radial non-increasing rearrangement of $|f|^q$. 

 Now, when $f \in L(n+2,1)$, we have from an equivalent characterization of Lorentz spaces that
 \begin{equation}\label{t}
   \int_{0}^{\infty} \left[ f^{**}(\rho) \rho^{\frac{q}{n+2}} \right]^{\frac{1}{q}} \ \frac{d\rho}{\rho} < \infty, \quad \text{for $q < n+2$}.
   \end{equation}
   In our subsequent analysis, we will consider the case when $q=n+1$ which ensures the validity of \eqref{t}. It thus follows that
   \begin{equation}\label{conv1}
    \mathbf{P}^f_q(x_0,t_0,r) \to 0\ \text{as $r \to 0$}
    \end{equation}
    when $f \in L(n+2, 1)$ and $q< n+2$.  
\end{rem}

We now fix a universal parameter which plays a crucial role in our compactness arguments.  Let $\beta>0$ be the optimal H\"older exponent such that any arbitrary $C^{2,1}$ viscosity solution $u$   of
\[
u_t=\Delta^N_p u\quad \text{is in $H^{1, \beta}_{loc}$}.
\]
The fact  that $\beta>0$ follows from the gradient H\"older continuity result established in  \cite{JS}. We then   fix some  $\alpha>0$ such that 
\begin{equation}\label{universal}
\alpha < \beta.
\end{equation} 
We refer to \cite[Chapter 4]{Li} for the precise notion of parabolic H\"older spaces $H^{k, \alpha}$. 

In our analysis, we will also need the following notion.
 \begin{defn}\label{alphadec}
 Given a modulus of continuity $\omega:[0, \infty) \to [0, \infty)$ and $\eta \in (0,1]$, we say that $\omega$ is $\eta$-decreasing if
 \[
 \frac{\omega(t_1)}{t_1^{\eta}} \geq \frac{\omega(t_2)}{t_2^{\eta}}, \text{for all $t_1 \leq t_2$.}\]
 \end{defn}

\subsection{Statement of the main results}
We now state our first  main result. This result corresponds to the  regularity  estimate  in the borderline  case, i.e.,  gradient continuity estimates with dependence on the  $L(n+2,1)$ norm of $f$. 

\begin{thm}\label{main}
For a given $p>1$, let $u$  be a $W^{2,1,n+1}$ viscosity solution of \eqref{m} in $Q_1$  where $f \in L(n+2,1)$. Then $\nabla u$ is continuous inside of $Q_1$. Moreover, the following  borderline estimates hold
\begin{equation}\label{bm}
\begin{cases}
|\nabla u(x_0,t_0)| \leq  C( \mathbf{P}^{f}_{n+1} (x_0, t_0, 1/2) + \|u\|_{L^{\infty}(Q_1)})\ \text{ \emph{for} $(x_0, t_0) \in Q_{1/2}$},
\\
|\nabla  u(x_1,t_1) - \nabla u(x_2,t_2) | \\ \leq C(n, p) ( \| u\|_{L^{\infty}(Q_{3/4})} |(x_1, t_1) - (x_2, t_2)|^{\alpha/4} +  \sup_{(x,t) \in Q_1}  \mathbf{P}^{f}_{n+1} (x,t,  4 |(x_1, t_1)-(x_2, t_2)|^{1/4})),\\
\operatorname{sup}_{\{(x, t_1), (x, t_2) \in Q_{1/2}: t_1 \neq t_2\}} \frac{|u(x, t_1) - u(x, t_2)|}{|t_1-t_2|^{1/2}} \leq C(\operatorname{sup}_{\{(x,t) \in Q_1\}} \mathbf{P}^{f}_{n+1} (x,t, 1/2) + \|u\|_{L^{\infty}(Q_1)}).\end{cases}
\end{equation}
  whenever  $(x_1,t_1), (x_2, t_2)  \in Q_{1/2}$,  and  where  $\alpha$ is as in \eqref{universal}.  \end{thm}
  
  Note that the second estimate in \eqref{bm} above provides a modulus of continuity for $\nabla u$ in view of \eqref{conv1}.
  
  In the case  $f \in L^{m}(\R^{n+1})$ with   $m>n+2$, we  obtain the following regularity result that improves Theorem 1.2 in \cite{APR}.
  
 \begin{thm}\label{main1}
 For $p>1$ and $m>n+2$, let $u$ be a $W^{2,1,m}$ viscosity solution of \eqref{m} in $Q_1$, where $f \in L^{m}$. Then, $\nabla u \in C^{\alpha_0}(\overline{Q_{1/2}})$ for some $\alpha_0=\alpha_0(n, p, m)$. Moreover, we have that the  following estimates hold
 \begin{equation}
 \begin{cases}
 \|\nabla u\|_{H^{\alpha_0}(Q_{1/2})} \leq C(n, p, \|f\|_{L^m}, \|u\|_{L^{\infty}(Q_1)}),
 \\
 \operatorname{sup}_{\{(x, t_1), (x, t_2) \in Q_{1/2}: t_1 \neq t_2\}} \frac{|u(x, t_1) - u(x, t_2)|}{|t_1-t_2|^{\frac{1+\alpha_0}{2}}}  \leq C(n, p, \|f\|_{L^m}, \|u\|_{L^{\infty}(Q_1)}).\end{cases}\end{equation} 
 \end{thm}

 %\begin{rem}
 %Note that in Theorem \ref{main1}, we don't need to assume apriori that $f$ is continuous unlike that in Theorem \ref{main} above. 
 %\end{rem}

\section{Proof of the main results}\label{mn}

\subsection{Proof of Theorem \ref{main}}

Before proceeding further, we would like to alert the reader that in what follows,  the number $q$ in Lemma \ref{app1}-Lemma \ref{rt}   equals $n+1$.

We now state our first relevant approximation lemma which plays a very crucial role in the separation of phases. This is the parabolic version of Lemma 3.1 in \cite{BM} and corresponds to the non-degenerate phase in our final iteration argument in the proof of Theorem \ref{main}. 

\begin{lem}\label{app1}
Let $u$ be a  $W^{2, 1, q}$ viscosity solution of
\begin{equation}\label{at}
\bigg(\delta_{ij} + (p-2) \frac{(\delta u_i+A_i)(\delta u_j +A_j)}{ |\delta \nabla u+A|^2} \bigg) u_{ij}-u_t=f\quad \text{in $Q_1$},
\end{equation}
 with $|u| \leq 1$ and $|A| \geq 1$.  Given $\tau>0$, there exists  $\delta_0=\delta_0(\tau)>0$  such that if 
\[
\delta, \left(\frac{1}{|Q_{3/4}|} \int_{Q_{3/4}}  |f|^q \right)^{1/q} \leq \delta_0,
\]
then $\|w-u\|_{L^{\infty}(Q_{1/2})} \leq \tau$, for some  $w \in C^{2,1}(\overline{Q_{1/2}})$ with universal $C^{2,1}$ bounds depending only on $n, p$ and  independent of $|A|$.
%\begin{equation}\label{to}
%\begin{cases}
%\bigg(\delta_{ij} + (p-2) \frac{A_i A_j}{  |A|^2} \bigg) w_{ij}=0\ \text{in $B_{3/4}$}
%\\
%w=u\ \text{on $\partial B_{3/4}$}
%\end{cases}
%\end{equation}

\end{lem}

\begin{proof}
We argue by contradiction. If not, then there exists $\tau_0>0$ and   a sequence of pairs  $\{u_k, f_k \}$ that solves \eqref{at} corresponding to $\{\delta_k, A_k\}$ with  $\delta_k \to 0, f_k \to 0\ \text{in $L^q(B_{3/4})$}$ as $k \to \infty$ and  such  that $u_k's$ are not $\tau_0$ close to any such $w$.  We note that the equation satisfied by $u_k$ can be rewritten as 
\begin{equation}\label{at1}
\bigg(\delta_{ij} +(p-2) \frac{(\tilde \delta_k (u_k)_i + (\tilde A_k)_i ) (\tilde \delta_k (u_k)_j + (\tilde A_k)_j)}{ |\tilde \delta_k \nabla u_k + \tilde A_k|^2}  \bigg) (u_k)_{ij}-(u_k)_t =f_k,
\end{equation}
where $\tilde \delta_k = \frac{\delta_k}{|A_k|}$ and $\tilde A_k= \frac{A_k}{|A_k|}$. Since $|A_k| \geq 1$, we have $\tilde \delta_k \to 0$ as $k \to \infty$.

Now it follows that  for each $k$, $u_k \in  \mathcal{S}(\lambda, \Lambda, f_k)$ with $\lambda, \Lambda$ as in \eqref{l} and moreover we have that $\|f_k\|_{L^{q}(Q_{3/4})}$ is uniformly bounded independent of $k$.  Thus from the Krylov-Safonov-type H\"older estimates as in \cite[Lemma 5.1]{CCS} (see also \cite{W1}), we have that  $u_k$'s are uniformly H\"older continuous in $\overline{Q_{3/5}}$. Therefore, upto a subsequence,  by Arzela-Ascoli we may assume that $u_k \to u_0$ uniformly on $Q_{3/5}$ and, moreover, we can also assume that  $\tilde A_k  \to A_0$ (by possibly passing to   another subsequence)  such that  $|A_0|=1$. 

We now  claim that $u_0$ solves 
\begin{equation}\label{lim}
\bigg(\delta_{ij} + (p-2) (A_0)_i (A_0)_j \bigg) (u_0)_{ij} -(u_0)_t=0.
\end{equation}
By standard theory, it suffices to check that $u_0$ is a $C^{2}-$viscosity solution to the above limiting equation.  We note that the stability result in Theorem 6.1 in \cite{CCS} cannot be directly applied here,  because of the singular dependence of the operator in the \say{gradient} variable. Similar to the elliptic case as in \cite{BM}, we thus argue as follows.

 Let $\phi$ be a $C^{2,1}$ function such that the graph of  $\phi$ strictly touches the graph of $u_0$ from above  at $(x_0, t_0) \in Q_{1/2}$. We show that  at $(x_0,t_0)$
\begin{equation}\label{et}
\bigg(\delta_{ij} + (p-2) (A_0)_i (A_0)_j \bigg) \phi_{ij}-\phi_t \geq 0.
\end{equation}
Suppose that  is not the case.  Then, there exists $\ve, \eta, r>0$  small enough such that
\begin{equation}\label{con2}
\begin{cases}
\bigg(\delta_{ij} + (p-2) (A_0)_i (A_0)_j \bigg) \phi_{ij}-\phi_t \leq - \ve\ \text{in $Q_r(x_0,t_0)$},
\\
\phi - u_0 \geq \eta\ \text{on $\partial_p Q_r(x_0,t_0)$}.
\end{cases}
\end{equation}
We now show that for every $k$, there exists a  perturbed test function $\phi + \phi_k$, with $\phi_k \in W^{2,1,q}$, such  that
\begin{equation}\label{con5}
F_k^{*} ( \nabla (\phi +\phi_k), \nabla^{2} (\phi + \phi_k))-(\phi + \phi_k)_t \leq f_k - \ve\ \text{in $Q_r(x_0,t_0)$},
\end{equation}
where $F_k^{*}$ is the upper semicontinuous relaxation of the operator  in \eqref{at1}. Moreover, we can also ensure that  $ (\phi+ \phi_k) - u_k$ has a minimum  in $Q_{r}(x_0, t_0) \setminus \partial_p Q_r(x_0, t_0)$ for large enough $k's$. This would  then contradict the viscosity formulation for $u_k$ for such $k's$ and hence \eqref{et} would follow. 

Therefore, under the assumption that \eqref{con2} holds, we now show the validity of \eqref{con5}.  We first observe that from \eqref{con2}, the following differential inequality holds
\begin{eqnarray}
\label{Fkineq}
F_k^{*} ( \nabla (\phi +\phi_k), \nabla^{2} (\phi + \phi_k)) - (\phi+\phi_k)_t & \leq & \mathcal{P}_{\lambda, \Lambda}^{+} (\nabla^2 \phi_k) -  (\phi_k)_t +C_0 |\tilde A_k - A_0|\\  
&&+ C_0 \tilde \delta_k |\nabla \phi_k| +  C_0 \tilde \delta_k |\nabla \phi| - \ve,
\notag
\end{eqnarray}
where $C_0=C_0(\|\nabla^2 \phi\|, p, n)$ and $\lambda, \Lambda$ are as in \eqref{l}.   Inequality \eqref{Fkineq} will follow by  adding and subtracting $\bigg(\delta_{ij} + (p-2) (A_0)_i (A_0)_j \bigg) \phi_{ij}$, by using \eqref{con2}, and then by splitting the considerations depending on whether $\Big|A_0 - (\tilde A_k + \tilde \delta_k (\nabla \phi + \nabla \phi_k))\Big| < 1/2$ or $> 1/2$. This   is similar to the argument as in  $(3.9)-(3.11)$ in the proof of  Lemma 3.1 in \cite{BM}. We nevertheless provide the details for the sake of completeness. 

\medskip

\emph{Case 1:} When $|A_0 - (\tilde A_k + \tilde \delta_k (\nabla \phi + \nabla \phi_k))| < 1/2$.

\medskip

In this case, we first note   by  triangle inequality that $|(\tilde A_k + \tilde \delta_k (\nabla \phi + \nabla \phi_k)) |>\frac{1}{2}$  since $|A_0| = 1$. Now  we have that  the function 
 \[
 a \to (p-2) \frac{ a_i a_j}{|a|^2},\ \text{for $|a|> 1/2$,}
 \]
 is  Lipschitz continuous, therefore it follows that
 \begin{align}\label{bt}
&\bigg(\delta_{ij} + (p-2) \frac{\tilde A_k + \tilde \delta_k ( \nabla \phi +\nabla \phi_k)}{| \tilde A_k +  \tilde \delta_k( \nabla \phi + \nabla \phi_k)|^2} \bigg) (\phi + \phi_k)_{ij}-  \bigg(\delta_{ij} + (p-2) (A_0)_i (A_0)_j \bigg) \phi_{ij}
\\
& \leq \mathcal{P}_{\lambda, \Lambda}^{+} (\nabla^2 \phi_k)  + C \|\nabla^2 \phi\|_{L^{\infty}} (  |\tilde A_k - A_0| +  \tilde \delta_k |\nabla \phi_k| +  \tilde \delta_k |\nabla \phi|).\notag
\end{align}

Thus  by  adding and subtracting $\bigg(\delta_{ij} + (p-2) (A_0)_i (A_0)_j \bigg) \phi_{ij}$ to $F_{k}^*( \nabla (\phi + \phi_k), \nabla^2(\phi + \phi_k)) - (\phi+\phi_k)_t$  and by using \eqref{con2} and \eqref{bt}, we observe that \eqref{Fkineq} follows in this case. 

\medskip

\emph{Case 2:} When $|A_0 - (\tilde A_k + \tilde \delta_k (\nabla \phi + \nabla \phi_k))| >1/2$.

\medskip

In this case, we note that since 
\begin{align}
& F_{k}^*( \nabla (\phi + \phi_k), \nabla^2(\phi + \phi_k))-  \bigg(\delta_{ij} + (p-2) (A_0)_i (A_0)_j \bigg) \phi_{ij} 
\\
& \leq \mathcal{P}_{\lambda, \Lambda}^{+} (\nabla^2 \phi_k) + C \|\nabla^2 \phi\|_{L^{\infty}} 
\notag
\\
& \leq  \mathcal{P}_{\lambda, \Lambda}^{+} (\nabla^2 \phi_k) + C\|\nabla^2 \phi\|_{L^{\infty}}|A_0 - (\tilde A_k + \tilde \delta_k (\nabla \phi + \nabla \phi_k))|\notag\\ & \text{ (using $|A_0 - (\tilde A_k + \tilde \delta_k (\nabla \phi + \nabla \phi_k))| >1/2$)}\notag
\\
& \leq  \mathcal{P}_{\lambda, \Lambda}^{+} (\nabla^2 \phi_k) + C\|\nabla^2 \phi\|_{L^{\infty}} ( |A_0 - \tilde A_k| + \tilde \delta_k |\nabla \phi| + \tilde \delta_k |\nabla \phi_k|),\ \text{ (by triangle inequality)}
\notag
 \end{align} 
 we again  obtain by  adding and subtracting $\bigg(\delta_{ij} + (p-2) (A_0)_i (A_0)_j \bigg) \phi_{ij}$ to $F_{k}^*( \nabla (\phi + \phi_k), \nabla^2(\phi + \phi_k)) - (\phi + \phi_k)_t$  and by using  \eqref{con2} that \eqref{Fkineq} follows in this case as well. 

\medskip

Subsequently, we let $\phi_k$ be  a strong solution to the following boundary value problem
\begin{equation*}
\begin{cases}
\mathcal{P}_{\lambda, \Lambda}^{+} (\nabla^2 \phi_k) + C_0 |\tilde A_k - A_0|  + C_0 \tilde \delta_k |\nabla \phi_k| +  C_0 \tilde \delta_k |\nabla \phi| -(\phi_k)_t= f_k\ \text{in $Q_r(x_0,t_0)$},
\\
\phi_k=0\ \text{on $\partial_p Q_r(x_0, t_0)$}.
\end{cases}
\end{equation*}
The existence of such strong $W^{2,1,q}$ solution is guaranteed by Theorem 2.8 in \cite{CCS}. Therefore, with such $\phi_k$, we  have that \eqref{con5} holds.  

We now  observe that since  $f_k \to 0$ in $L^{q}$ and also $\tilde \delta_k , |\tilde A_k - A_0| \to 0$, from the generalized maximum principle as in \cite[Proposition 2.6]{CCS},  we have  that 
\[
\|\phi_k\|_{L^{\infty}(Q_r(x_0, t_0))} \to 0\ \text{as $k \to \infty$}.
\]
Since $\phi - u_0$ has a strict  minimum at $(x_0, t_0)$ in $Q_{1/2}$ and consequently in $Q_r(x_0, t_0)$,  it follows for large $k's$ that  $ (\phi+ \phi_k) - u_k$ has  a  minimum in  $Q_r(x_0, t_0) \setminus \partial_p Q_r(x_0, t_0)$ (since $\phi_k \equiv 0$ on $\partial_p Q_r(x_0, t_0)$ and $\phi-u_0 \geq \eta$ on $\partial_p Q_r(x_0, t_0)$). From this and  in view of the arguments immediately after \eqref{con5} above, we can assert that  \eqref{et} follows. 

Now  by an analogous argument, we can assert that the opposite inequality holds in \eqref{et} in the situation when  the graph of $\phi$  touches the graph of  $u_0$ from below at $x_0$ and consequently it  follows that $u_0$ solves \eqref{lim}. Moreover, since $|u_0| \leq 1$,  we have from the classical theory  that $u_0$  is smooth with universal $C^{2,1}$ bounds in $Q_{1/2}$.  This would then be a  contradiction for large enough $k'$s since $u_k \to u_0$ uniformly.  This finishes the proof of the lemma.

\end{proof}

As a  consequence of Lemma \ref{app1}, we have the following result on the  affine approximation of $u$ at $(0,0)$,  provided there is a sufficiently large non-degenerate slope  at a certain scale.   As the reader will see, this is ensured by the fast geometric convergence of the approximations. 

\begin{lem}\label{ap2}
Let $u$ be  a viscosity  solution of
\[
\bigg(\delta_{ij}+ (p-2) \frac{u_i u_j}{|\nabla u|^2} \bigg) u_{ij}-u_t=f\quad \text{in $Q_1$},
\]
with  $u(0,0)=0$.  There exists a universal $\delta_0>0$ such that if for some  $A \in \R^n$, satisfying  $M \geq|A| \geq 2$,  we have 
\[
\|u- \langle A,x\rangle \|_{L^{\infty}(Q_1)} \leq \delta_0,
\]
and also 
%\[
%\int_{0}^{1}  \bigg( \def\avint{\mathop{\,\rlap{-}\!\!\int}\nolimits} \avint_{Q_s} |f|^q\bigg)^{1/q} ds \leq \delta_0^2,
%\]
\[
\int_{0}^{1}  \bigg(\aver{Q_s} |f|^q\bigg)^{1/q} ds \leq \delta_0^2,
\]
then there exists an affine function  $ L_0$ with universal bounds depending also on $M$ such that
\begin{equation}\label{difr}
|u(x,t)-  L_0(x) | \leq C(|x|^2+|t|)^{\frac{1}{2}} K((|x|^2+|t|)^{\frac{1}{2}} ).
\end{equation}
Here 
%\[
%K(r)\overset{def}= r^{\alpha/2} +  \int_{0}^{r^{1/2}} (\def\avint{\mathop{\,\rlap{-}\!\!\int}\nolimits} \avint_{Q_s} |f|^q)^{1/q} ds
%\]
\[
K(r)\overset{def}= r^{\alpha/2} +  \int_{0}^{r^{1/2}} \left(\aver{Q_s} |f|^q\right)^{1/q} ds
\]
and $\alpha$ is the universal parameter as in \eqref{universal}. In view of Remark \ref{mod}, we note that for $f \in L(n+2, 1)$,  we have that $K(r) \to 0$ as $r \to 0$. 

\end{lem}

\begin{proof}

 We will show that  for  for every $k=0,1,2, \ldots$, there exist linear functions \linebreak $\tilde  L_k x\overset{def}= \langle A_k, x\rangle $  such that
\begin{equation}\label{cl}
\begin{cases}
\|u-\tilde L_k\|_{L^{\infty}(Q_{r^k})} \leq  r^k \omega(r^k),
\\
|A_k- A_{k-1}| \leq C  \omega(r^{k-1}),
\end{cases}
\end{equation}
for some $r<1$  universal, independent of $\delta_0$.  Here we let for a given $k$,

\begin{equation}\label{om}
\omega(r^k)= \frac{1}{\delta_0} \sum_{i=0}^k r^{i \alpha} \omega_1\left(\frac{3}{4}r^{k-i}\right),
\end{equation}
with  $\omega_1$ defined in the following way
%\[
%\omega_1(s)=  \max\ \left( s\left( \def\avint{\mathop{\,\rlap{-}\!\!\int}\nolimits} \avint_{Q_s
%} |f|^q\right)^{1/q}, \delta_0^2 \frac{4}{3} s\right).
%\]
\[
\omega_1(s)=  \max\ \left( s\left(\aver{Q_s} |f|^q \right)^{1/q}, \delta_0^2 \frac{4}{3} s\right).
\]
We note that $\delta_0$ is to be fixed later.  We also let $A_0\overset{def}= A$.  Now, suppose $A_k$ exists  upto some $k$ with the bounds as in \eqref{cl}. Then, we observe that 
\begin{align}\label{ndeg}
|A_k| &\geq |A_0| -  ( |A_1-  A_0| +\ldots+ |A_k- A_{k-1}|) 
\\
& > 2 -  C  \sum \omega(r^i) > 2 -   \frac{C}{\delta_0} \sum \omega_1 \left(\frac{3}{4}r^i\right)\ \left(\text{using the Cauchy product formula}\right)
\notag
\\
& \geq 2 - C_1 \delta_0 > 1\  \left(\text{if $\delta_0$ is small enough}\right).\notag
\end{align}
In the last inequality in \eqref{ndeg} above  we also used the fact that
%\begin{align}\label{2.0}
% \sum \omega_1\left(\frac{3}{4} r^i\right) &\leq C\left( \delta_0^2 \sum r^i +  \sum  \frac{3 r^{i}}{4} \left( \def\avint{\mathop{\,\rlap{-}\!\!\int}\nolimits} \avint_{Q_{\frac{3r^i}{4}}} |f|^q\right)^{1/q}  \right)
%\\
%& \leq C\left( \delta_0^2  +  \int_{0}^{1}  \left( \def\avint{\mathop{\,\rlap{-}\!\!\int}\nolimits} \avint_{Q_s} |f|^q\right)^{1/q} ds \right) \leq C_2 \delta_0^2.
%\notag
%\end{align}
\begin{align}\label{2.0}
 \sum \omega_1\left(\frac{3}{4} r^i\right) &\leq C\left( \delta_0^2 \sum r^i +  \sum  \frac{3 r^{i}}{4} \left(\aver{Q_{\frac{3r^i}{4}}} |f|^q\right)^{1/q}  \right)
\\
& \leq C\left( \delta_0^2  +  \int_{0}^{1}  \left(\aver{Q_s} |f|^q\right)^{1/q} ds \right) \leq C_2 \delta_0^2.
\notag
\end{align}

Note that the last inequality in \eqref{2.0} is a consequence of the following  estimate 
%\[
%\sum  \frac{3 r^{i}}{4} \left( \def\avint{\mathop{\,\rlap{-}\!\!\int}\nolimits} \avint_{Q_{\frac{3r^i}{4}}} |f|^q\right)^{1/q}  \leq C \int_{0}^{1}  \left( \def\avint{\mathop{\,\rlap{-}\!\!\int}\nolimits} \avint_{Q_s} |f|^q\right)^{1/q} ds,
%\]
\[
\sum  \frac{3 r^{i}}{4} \left(\aver{Q_{\frac{3r^i}{4}}} |f|^q\right)^{1/q}  \leq C \int_{0}^{1}  \left(\aver{Q_s} |f|^q\right)^{1/q} ds,
\]
which in turns  follows by  breaking the integral in the above expression  into integrals  over dyadic  subintervals of the type $[\frac{3}{4} r^{i}, \frac{3}{4} r^{i-1}]$. 

Thus the estimate in \eqref{ndeg}   ensures that the non-degeneracy condition in Lemma \ref{app1} holds for every $k$. We prove the claim in \eqref{cl} by induction.  From the hypothesis of the lemma, the case when $k=0$ is easily verified  with $A_0=A$ with our choice of $\omega$.  Let us now assume that the claim as in \eqref{cl} holds upto some $k$. We then  consider  
\[
v= \frac{(u - \tilde  L_k)(r^k x, r^{2k}t)}{ r^k \omega(r^k)},
\]
which solves
\begin{equation}
\bigg(\delta_{ij} + (p-2) \frac{(\omega(r^k) v_i + (A_k)_i)(  \omega(r^k) v_j + (A_k)_j)}{| \omega(r^k) \nabla v+  A_k|^2} \bigg) v_{ij}-v_t= \frac{r^k}{ \omega(r^k)} f(r^k x, r^{2k}t).
\end{equation}
For ease of notation, by $\tilde{L}(r^k x, r^{2k}t)$ we mean $\tilde{L}(r^kx)$ and from now on we will use this notation provided that there is no ambiguity with it.  Now, by a change of variable formula and the definition of $\omega$ it follows that,  with
\[
f_k(x, t)\overset{def}=  \frac{r^k}{ \omega(r^k)} f(r^k x, r^{2k}t),
\]
we have 
%\begin{align}\label{com1}
% \bigg(\frac{1}{|Q_{3/4}|} \int_{Q_{3/4}}  |f_k|^q \bigg)^{1/q} &= \frac{r^k}{ \omega(r^k)} \bigg(\frac{1}{|Q_{3r^{k}/4 }|} \int_{Q_{\frac{3 r^k}{4}}} |f(y,s)|^q dyds \bigg) ^{1/q}
%\\
%& \leq \frac{r^k}{\omega_1(\frac{3 r^k}{4}) \frac{1}{\delta_0}} \bigg(\frac{1}{|Q_{3r^{k}/4 }|} \int_{Q_{\frac{3 r^k}{4}}} |f(y,s)|^q dyds \bigg) ^{1/q}
%\notag
%\\
%& \leq \frac{ r^{k} \bigg(\frac{1}{|Q_{3r^{k}/4 }|} \int_{Q_{\frac{3 r^k}{4}}} |f(y,s)|^q dy ds\bigg) ^{1/q}}{ \frac{3 r^{k}}{4 \delta_0}  \bigg( \def\avint{\mathop{\,\rlap{-}\!\!\int}\nolimits} \avint_{Q_{\frac{3 r^k}{4}}} |f(y,s)|^q  
%dyds\bigg)^{1/q}}
%\notag
%\\
% &\leq \frac{4}{3} \delta_0.
%\notag
%\end{align}
\begin{align}\label{com1}
 \bigg(\frac{1}{|Q_{3/4}|} \int_{Q_{3/4}}  |f_k|^q \bigg)^{1/q} &= \frac{r^k}{ \omega(r^k)} \bigg(\frac{1}{|Q_{3r^{k}/4 }|} \int_{Q_{\frac{3 r^k}{4}}} |f(y,s)|^q dyds \bigg) ^{1/q}
\\
& \leq \frac{r^k}{\omega_1(\frac{3 r^k}{4}) \frac{1}{\delta_0}} \bigg(\frac{1}{|Q_{3r^{k}/4 }|} \int_{Q_{\frac{3 r^k}{4}}} |f(y,s)|^q dyds \bigg) ^{1/q}
\notag
\\
& \leq \frac{ r^{k} \bigg(\frac{1}{|Q_{3r^{k}/4 }|} \int_{Q_{\frac{3 r^k}{4}}} |f(y,s)|^q dy ds\bigg) ^{1/q}}{ \frac{3 r^{k}}{4 \delta_0}  \bigg(\aver{Q_{\frac{3 r^k}{4}}} |f(y,s)|^q  
dyds\bigg)^{1/q}}
\notag
\\
 &\leq \frac{4}{3} \delta_0.
\notag
\end{align}
Moreover
\[
\omega(r^k) \leq \sum \omega(r^i) \leq C_0 \delta_0.
\]

Therefore, $v$ satisfies an equation for which the conditions in Lemma \ref{app1} are satisfied. Consequently for a given $\tau>0$, we can find $\delta_0>0$ such that for some $w$ with universal $C^{2,1}$ bounds  we have that $\|w-v\|_{L^{\infty}(Q_{1/2})} \leq \tau$. Now, given that $w$ has uniform $C^{2,1}$ bound, there exists a  universal $C>0$ such that
\begin{equation}\label{ert0}
|w(x,t)- w(0,0)- Lx| \leq C\left(|x|^2+|t|\right),
\end{equation}
where $L$ is the linear approximation for $w$ at $(0,0)$. We then choose $r>0$  small enough such that
\begin{equation}\label{ert1}
Cr^2= \frac{r^{1+\alpha}}{2},
\end{equation}
where $\alpha$ is as in \eqref{universal}. Subsequently, we  let $\tau= \frac{r^{1+\alpha}}{4}$ which decides the choice of $\delta_0$.  Then using \eqref{ert0} and \eqref{ert1},   by an application of triangle inequality we have
\begin{align}\label{ert}
 \|v- L\|_{L^{\infty}(Q_r)} &\leq \|w-v\|_{L^{\infty}(Q_r)} + \|w-w(0,0) -Lx\|_{L^{\infty}(Q_r)} +|w(0,0)|\\
& \leq \frac{r^{1+\alpha}}{2} + 2\tau \leq r^{1+\alpha}.\notag \end{align}
Note that in \eqref{ert} above, we used that $|w(0,0)| \leq \tau$ which is a consequence of the fact that $v(0,0)=0$ and  also that  $\|w-v\|_{L^{\infty}(Q_{1/2})} \leq \tau$.

 Consequently, by scaling back to $u$  we obtain
\begin{equation}\label{sc}
\|u - \tilde  L_{k+1}\|_{L^{\infty}(Q_{r^{k+1}})} \leq  r^{k+1} r^{\alpha} \omega(r^k) \leq  r^{k+1} \omega(r^{k+1}),
\end{equation}
where  $\tilde L_{k+1}(x)\overset{def}= \tilde L_k +  r^k \omega(r^k) L\bigg(\frac{x}{r^k}  \bigg)$. Note that in the last inequality in \eqref{sc} we also  used the following $\alpha-$decreasing property of $\omega$ ( see Definition \ref{alphadec})
\begin{equation}\label{dec}
r^{\alpha} \omega(r^k) \leq \omega(r^{k+1}),
\end{equation}
which is easily seen from the expression of $\omega$ as in \eqref{om}.  This verifies the induction step. The conclusion now  follows by a standard  real analysis argument as in the proof of Lemma 4.9 in  \cite{AB}. 
\end{proof}

The next result  is an  improvement of flatness result that allows to handle the case when the affine approximation have  small slopes  at a   \say{$k$th-step}.  This corresponds to the degenerate alternative    in the iterative argument  in the proof of the main result Theorem \ref{main}.     
\begin{lem}\label{rt}
Let $u$ be a solution to
\begin{equation}\label{j1}
\bigg(\delta_{ij}+ (p-2) \frac{u_i u_j}{|\nabla u|^2} \bigg) u_{ij}-u_t=f\quad \text{in $Q_1$},
\end{equation}
 with $|u| \leq 3$ and $u(0,0)=0$.  There exists  a universal $\ve_0>0$ such that if 
%\begin{equation}\label{sml1}
%\int_{0}^{1}  \left( \def\avint{\mathop{\,\rlap{-}\!\!\int}\nolimits} \avint_{Q_s} |f|^{q}\right)^{\frac{1}{q}} ds \leq \ve_0,
%\end{equation}
\begin{equation}\label{sml1}
\int_{0}^{1}  \left( \aver{Q_s} |f|^{q}\right)^{\frac{1}{q}} ds \leq \ve_0,
\end{equation}
then there exists an affine function $L$,  with universal bounds,  and a universal $\eta \in (0,1)$  such that
\[
\|u- L\|_{L^{\infty}(Q_{\eta})} \leq \delta_0 \eta^{1+\alpha}.
\]
Here $\delta_0>0$ is as in Lemma \ref{ap2} above. Without loss of generality we may take $0<\ve_0 < \delta_0^2$. 

\end{lem}

\begin{proof}

We first show that given $\kappa>0$, there exists $\ve_0>0$ such that if  $u$ solves \eqref{j1} and $f$ satisfies the bound in \eqref{sml1}, then there exists a  $C^{2,1}$ viscosity solution $w$ to  the normalized $p$-parabolic equation \eqref{nplap} such that
\begin{equation}\label{close}
\|w-u\|_{L^{\infty}(Q_{1/2})} \leq  \kappa.
\end{equation}
Assume that \eqref{close} actually holds. It  then follows from the  $H^{1, \beta}$ regularity result   in \cite{JS} that there exists an affine function $L$ such that
\[
|w(x,t) -L(x)| \leq C(|x|^2+|t|)^{\frac{1+\beta}{2}}.
\]
We now choose  $\eta>0$ such that 
\[
C\eta^{1+\beta}= \frac{\delta_0}{2} \eta^{1+\alpha}\ \left(\text{This  crucially uses $\alpha < \beta$}\right).
\]
Subsequently, we choose $\kappa= \frac{\delta_0}{2} \eta^{1+\alpha}$, and this decides the choice of $\ve_0$. The conclusion of the lemma now  follows by an application of the triangle inequality. 

We are now going to prove \eqref{close}.  Suppose on the contrary,  \eqref{close} does not hold. Then there exists $\kappa_0>0$ and a sequence of pairs  $\{u_k, f_k\}$ which solves \eqref{j1} such that $u_k$ is not $\kappa_0$ close to any such $w$.  Notice that  \eqref{sml1} implies, for each $k\in \mathbb{N}$, that
\[
\|f_k\|_{L^{q}(Q_{3/4})} < \frac{C}{k}.
\] 
  Then, from uniform Krylov-Safonov-type  H\"older estimates as in \cite[Lemma 5.1]{CCS} (see also \cite{W1}) and Arzela-Ascoli, it follows that $u_k \to u_0$ uniformly in $Q_{\frac{1}{2}}$ upto a sub-sequence.  We now make the  claim that $u_0$ solves \eqref{nplap}., i.e. 
\begin{equation}
\label{u0ppar}
\bigg(\delta_{ij}+ (p-2) \frac{(u_0)_i (u_0)_j}{|\nabla u_0|^2} \bigg) (u_0)_{ij}-(u_0)_t=0\quad \text{in $Q_{\frac{1}{2}}$}
\end{equation}
 in the viscosity sense.  Once the claim is established, this would then be a contradiction for large enough $k$'s and thus \eqref{close} would follow. 

The proof is similar to that of the \emph{Claim} in  Lemma \ref{app1}. As before, we note that the stability  result in Theorem 6.1 in \cite{CCS} cannot be directly applied because the operator $\Delta_{p}^N$ does not satisfy the structural assumptions in \cite{CCS} because of singular dependence in the \say{gradient} variable.

  Let $\phi$ be a $C^{2,1}$ test function which   strictly touches the graph of  $u_0$ from above at some point  $(x_0, t_0) \in Q_{1/2}$.  In view of the equivalent characterization of viscosity solutions to \eqref{nplap} as in Lemma \ref{equiv1}, it suffices to consider the following two cases:
  \begin{enumerate}
  \item $\nabla \phi(x_0, t_0) \neq 0$, and
  \item $\nabla \phi(x_0, t_0)=0,  \nabla^2 \phi(x_0, t_0) = 0$.
\end{enumerate}   
For Case (1), we show that 
\begin{equation}\label{cl0}
\Delta_{p}^N \phi (x_0,t_0) -\phi_t(x_0, t_0) \geq 0.
\end{equation}
Suppose such is not the case. Then  there exists $\ve,r, \delta>0$ small enough such that
\begin{equation}\label{l1}
\begin{cases}
\Delta_{p}^N \phi(x) - \phi_t \leq - \ve\ \text{in $Q_r(x_0, t_0)$},
\\
\phi - u_0 > \delta\ \text{on $\partial Q_r(x_0, t_0)$}.
\end{cases}
\end{equation}
Moreover, we can also assume that in $Q_r(x_0, t_0)$, we have that
\begin{equation}\label{ph}
|\nabla \phi | \geq \kappa>0.
\end{equation}

We now show that for every $k$, there exists   perturbed test functions $\phi + \phi_k$ with $\phi_k \in W^{2,1,q}_{loc} (Q_r(x_0, t_0))$ such  that
\begin{equation}\label{con}
F^{*} ( \nabla (\phi +\phi_k), \nabla^{2} (\phi + \phi_k)) - (\phi + \phi_k)_t \leq f_k - \ve\ \text{in $Q_r(x_0, t_0)$},\ \text{ with $F^{*}$ as in \eqref{up}}.
\end{equation}
Moreover, we can also ensure that  $ (\phi+ \phi_k) - u_k$ has a minimum  in $Q_{r}(x_0,t_0)$ for large enough $k's$. This would  then contradict the viscosity formulation for $u_k$, and hence \eqref{cl0} would follow.     

Hence, under the assumption that \eqref{l1} is valid, we now turn our attention to establish  \eqref{con}. We first observe that  because of \eqref{l1},  \eqref{ph},  the following inequality holds,
\begin{align}\label{s1}
& F^{*} ( \nabla (\phi +\phi_k), \nabla^{2} (\phi + \phi_k)) - (\phi +\phi_k)_t  \leq \mathcal{P}^{+}_{\lambda, \Lambda} (\nabla^2 \phi_k) - (\phi_k)_t + C(\kappa, \|\nabla^2 \phi\|) |\nabla \phi_k| -\ve,
\end{align}
with $\lambda, \Lambda$ as in \eqref{l}.
Here $\mathcal{P}_{\lambda, \Lambda}^{+}$ is the maximal Pucci operator defined as   in \eqref{max}. This inequality again follows by an argument similar to that  used in deriving \eqref{Fkineq} in the proof of Lemma \ref{app1}   by adding and subtracting $\Delta_{p}^N \phi$, by using \eqref{l1} and then by splitting considerations depending on whether
\[
|\nabla \phi_k| < \kappa/2\ \text{or}\ > \kappa/2.
\]
At this point, given $k$, we   look for $\phi_k$ which is a strong solution to
\begin{equation}\label{snewone}
\begin{cases}
 \mathcal{P}^{+}_{\lambda, \Lambda} (\nabla^2 \phi_k)  + C(\kappa, \|\nabla^2 \phi\|) |\nabla \phi_k|-( \phi_k)_t= f_k\ \text{in $Q_r(x_0, t_0)$},
\\
\phi_k= 0\ \text{on $\partial_p Q_r(x_0, t_0)$}.
\end{cases}
\end{equation}
Over here, we remind the reader that $\lambda, \Lambda$ is as in \eqref{l}. The existence of such strong solutions is again guaranteed by Theorem 2.8 in \cite{CCS}. Moreover since $f_k \to 0$ in $L^{q}$, therefore  from the generalized maximum principle we have that 
\[
\|\phi_k\|_{L^{\infty}(Q_r(x_0,t_0))} \to 0\ \text{as $k \to \infty$}.
\]

Now, since $\phi - u_0$ has a strict  minimum at $(x_0, t_0)$, it follows that for large $k's$ that  $ (\phi+ \phi_k) - u_k$ would have a minimum in the inside of $Q_r(x_0,t_0)$ (since $\phi_k \equiv 0$ on $\partial Q_r(x_0, t_0)$ and $\phi-u_0 > \delta$ on $\partial Q_r(x_0, t_0)$). However, because  of \eqref{s1} and \eqref{snewone}  we also  have that  \eqref{con} holds which violates the viscosity formulation for $u_k$'s for large enough $k's$.  Thus \eqref{cl0} holds in this case.

 For Case (2),  we only need to show that 
 \begin{equation}\label{cl00}
 \phi_t(x_0, t_0) \leq 0,
\end{equation}

Suppose that \eqref{cl00} does not hold. Then from the continuity of the derivatives of $\phi$, it follows that there exists $\gamma, \delta, r>0$  such that 
\begin{align*}
\phi_t(x_0,t_0)&\geq \gamma \quad \text{in $Q_r(x_0,t_0)$},\\
\varphi-u_0& >\delta \quad \text{on $\partial_p Q_r(x_0,t_0)$ and}\\
\mathcal{P}^{+}_{\lambda, \Lambda} (\nabla^2 \phi) & < \frac{\gamma}{2} \quad \text{in $Q_r(x_0,t_0)$, noting that this can be ensured since $\nabla^2 \phi(x_0, t_0)=0$}.
\end{align*}
For a given $k$, we  consider $\phi_k$, which is a strong solution to
\begin{equation}\label{s}
\begin{cases}
 \mathcal{P}^{+}_{\lambda, \Lambda} (\nabla^2 \phi_k)  -(\phi_k)_t= f_k\ \text{in $Q_r(x_0, t_0)$},
\\
\phi_k= 0\ \text{on $\partial Q_r(x_0, t_0)$}.
\end{cases}
\end{equation}
As before,  from the generalized maximum principle we have that 
\[
\|\phi_k\|_{L^{\infty}(Q_r(x_0,t_0))} \to 0\ \text{as $k \to \infty$}.
\]
Then,
\begin{eqnarray*}
F^{*} ( \nabla (\phi +\phi_k), \nabla^{2} (\phi + \phi_k)) - (\phi + \phi_k)_t & \leq & \mathcal{P}^{+}_{\lambda, \Lambda} (\nabla^2 \phi) -\phi_t\\
&&+ \mathcal{P}^{+}_{\lambda, \Lambda} (\nabla^2 \phi_k) -(\phi_k)_t\\
&\leq &\frac{\gamma}{2} -\gamma +f_k\leq f_k-\frac{\gamma}{2}.
\end{eqnarray*}

As before, since $\phi - u_0$ has a strict  minimum at $(x_0, t_0)$, it follows that for large $k's$ that  $ (\phi+ \phi_k) - u_k$ would have a minimum in $Q_r(x_0,t_0) \setminus \partial_p Q_r(x_0, t_0)$. On the other hand,
\[
\limsup_{(x,t) \to (x_0,t_0)}\left( F^{*} ( \nabla (\phi+\phi_k)(x,t), \nabla^2 (\phi+\phi_k)(x,t))-(\phi+\phi_k)_t - f_k(x,t)\right) \leq -\frac{\gamma}{2}.
\]
This is a contradiction to the viscosity formulation for all such $u_k$'s. Thus \eqref{cl00} holds in this case  and from Lemma \ref{equiv1} we can now assert that $u_0$ is a  viscosity subsolution  to \eqref{nplap}. In an analogous way, we can show that $u_0$ is a  viscosity supersolution to \eqref{nplap}  and consequently in view of the arguments after \eqref{u0ppar}, the conclusion follows.
\end{proof}

With this Lemma \ref{ap2} and Lemma \ref{rt} in hand, we now proceed with the proof of our main result. For notational convenience,  from now on, sometimes we will denote a point $(x,t)$ in space time by $X$. Also we set $|X| \overset{def}= \text{max} (|x|, |t|^{1/2})$. Note that $|X| \approx |x| + |t|^{1/2}.$

\begin{proof}[Proof of Theorem \ref{main}]

It suffices to establish the following affine approximation for $u$ at $(0,0)$. More precisely, we will show there exists an affine function $\tilde L$ such that
\begin{equation}\label{des}
|u(x,t)- \tilde L(x)| \leq C|X| K_0(4 |X|),\ (x,t) \in Q_{1/2},
\end{equation}
 where $K_0(|X|)$ is defined as 
%\[
%K_0(|X|) \overset{def}= \bigg(  \int_0^1  \left( \def\avint{\mathop{\,\rlap{-}\!\!\int}\nolimits} \avint_{Q_s} |f|^q\right)^{1/q} ds  \bigg) |X| ^{\alpha/4} +    C_0(\alpha) \int_{0}^{|X|^{1/4}} \left( \def\avint{\mathop{\,\rlap{-}\!\!\int}\nolimits} \avint_{Q_s} |f|^q\right)^{1/q} ds,
%\]
\[
K_0(|X|) \overset{def}= \bigg(  \int_0^1  \left(\aver{Q_s} |f|^q\right)^{1/q} ds  \bigg) |X| ^{\alpha/4} +    C_0(\alpha) \int_{0}^{|X|^{1/4}} \left( \aver{Q_s} |f|^q\right)^{1/q} ds,
\]
and where $C$ is  some universal constant.

Likewise a similar affine approximation holds at all points in $Q_{1/2}$ and consequently the estimates in \eqref{bm} follow by a standard real analysis argument. 

We may also  assume that $u(0,0)=0$.  Now with  $\eta, \ve_0$ as in Lemma \ref{rt} and  $\delta_0$ as in  Lemma \ref{ap2},   assume the following hypothesis for a given $i \in \mathbb{N}$,
\begin{equation*}\label{H}[H]
\begin{cases}
\text{There exists affine function $L_i(x)\overset{def}= \langle B_i, x\rangle $ such that}\  \|u-L_i\|_{L^{\infty}(Q_{\eta^i})}  \leq \delta_0 \eta^i  \omega( \eta^i)
\\
\text{and}\ |B_i|  \leq 2  \omega(\eta^i).
\\
\end{cases}
\end{equation*}
Here $\omega$ is defined instead as 
\begin{equation}\label{om1}
\omega(\eta^k)\overset{def}= \frac{1}{\ve_0} \sum_{i=0}^k \eta^{i \alpha} \omega_1(\eta^{k-i}),
\end{equation}
where we let $\omega_1$ to be
%\[
%\omega_1(r)\overset{def}=  \max\left( \int_{0}^r\left( \def\avint{\mathop{\,\rlap{-}\!\!\int}\nolimits} \avint_{Q_s
%} |f|^q\right)^{1/q}ds,   r\right).
%\]
\[
\omega_1(r)\overset{def}=  \max\left( \int_{0}^r\left(\aver{Q_s} |f|^q\right)^{1/q}ds,   r\right).
\]
 %$\tilde K(r)= \frac{1}{\ve_0} K_4(r)$  with $K_4$  being a $\alpha-$ decreasing modulus of continuity which  bounds  $K(r^2)$ with $K$   as in  Lemma \ref{ap2}. Moreover since $K(r^2)= r^{\alpha} + \int_{0}^{r} ( \def\avint{\mathop{\,\rlap{-}\!\!\int}\nolimits} \avint_{B_s} |f|^q)^{1/q} ds$, therefore  thanks to the   estimate  \eqref{int1}, we can assume that  $K_4(r^{1/2})$ is upperbounded by $\tilde K_0(r)$. 

By multiplying $u$ with a suitable constant we can assume that the Statement $[H]$ holds when $i=0$ with $L_0=0$.  Let $k$ be the first integer such that the Statement $[H]$  breaks. Then there  are two possibilities.

\emph{Case 1:} Suppose $k=\infty$. Then  given $X=(x,t)$, let $i\in \mathbb{N}$ be such that $|X| \sim \eta^{i}$. Then from the inequalities in $[H]$ and triangle inequality, it follows that
\begin{equation}\label{g1}
|u(x,t)| \leq |u(x,t) - L_i (x)| +  |L_i (x)| \leq C_1 \eta^{i} \omega(\eta^i) \leq C |X| \omega( 2|X|) \leq  C |X| K_0( 4 |X|),
\end{equation}
and thus \eqref{des} follows with $\tilde L=0$. The last inequality in \eqref{g1} is seen as follows:
%\begin{align}\label{b0}
% \omega(\eta^i)&= \frac{1}{\ve_0} \sum_{j=0}^i \eta^{j \alpha} \omega_1(\eta^{i-j})
%\\
%& \leq C \omega_1(\eta^{i/2}) \sum_{j=0}^{i/2}   \eta^{j\alpha} +  C \omega_1(1) \sum_{j=i/2}^{i} \eta^{j\alpha}\quad (\text{here we use $\omega_1$ is increasing})
%\notag
%\\
%& \leq C \bigg(  \int_0^1  \left( \def\avint{\mathop{\,\rlap{-}\!\!\int}\nolimits} \avint_{Q_s} |f|^q \right)^{1/q} ds  \bigg) \eta ^{i\alpha/2} +    C_0(\alpha) \int_{0}^{\eta^{i/2}} \left( \def\avint{\mathop{\,\rlap{-}\!\!\int}\nolimits} \avint_{Q_s} |f|^q\right)^{1/q} ds
%\notag
%\\
%& \leq C K_0( 4 |X|)\quad (\text{using $|X| \sim \eta^{i}$}).
%\notag
%\end{align}
\begin{align}\label{b0}
 \omega(\eta^i)&= \frac{1}{\ve_0} \sum_{j=0}^i \eta^{j \alpha} \omega_1(\eta^{i-j})
\\
& \leq C \omega_1(\eta^{i/2}) \sum_{j=0}^{i/2}   \eta^{j\alpha} +  C \omega_1(1) \sum_{j=i/2}^{i} \eta^{j\alpha}\quad (\text{here we use $\omega_1$ is increasing})
\notag
\\
& \leq C \bigg(  \int_0^1  \left(\aver{Q_s} |f|^q \right)^{1/q} ds  \bigg) \eta ^{i\alpha/2} +    C_0(\alpha) \int_{0}^{\eta^{i/2}} \left(\aver{Q_s} |f|^q\right)^{1/q} ds
\notag
\\
& \leq C K_0( 4 |X|)\quad (\text{using $|X| \sim \eta^{i}$}).
\notag
\end{align}

\emph{Case 2:} Suppose  instead that $k < \infty$. Then we have that the Statement $[H]$ is satisfied upto $k-1$. Now let
\[
v(x,t)\overset{def}= \frac{u(\eta^{k-1}x, \eta^{2(k-1)} t)}{\eta^{k-1} \omega(\eta^{k-1})},
\]
which solves
\[
\bigg( \delta_{ij} +(p-2)\frac{v_i v_j}{|\nabla v|^2} \bigg) v_{ij} - v_t= \frac{\eta^{k-1} f(\eta^{k-1} x, \eta^{2(k-1)}t)}{\omega(\eta^{k-1})}.
\]
Moreover, from  the estimates in [H] for $i=k-1$ it follows that $|v| \leq 2+ \delta_0 \leq 3$.  Also by change of variable, we have that for 
\[
f_k(x,t)= \frac{\eta^{k-1} f(\eta^{k-1} x, \eta^{2(k-1)} t)}{\omega(\eta^{k-1})}
\]
the following holds,  
%\begin{align}\label{d1}
%& \int_{0}^{1}  \left( \def\avint{\mathop{\,\rlap{-}\!\!\int}\nolimits} \avint_{Q_s} |f_k|^q\right)^{1/q} ds
%\\
%& \leq \ve_0  \frac{ \eta^{k-1} \int_{0}^{1}  \left( \def\avint{\mathop{\,\rlap{-}\!\!\int}\nolimits} \avint_{Q_s} |f(\eta^{k-1} x, \eta^{2(k-1)}t)|^q dxdt\right)^{1/q} ds}{  \int_{0}^{\eta^{k-1}}  {\left( \def\avint{\mathop{\,\rlap{-}\!\!\int}\nolimits} \avint_{Q_s} |f|^q \right)^{1/q}}}\notag
%\\
%& = \ve_0  \frac{ \int_{0}^{\eta^{k-1}}  ( \def\avint{\mathop{\,\rlap{-}\!\!\int}\nolimits} \avint_{Q_s} |f|^q )^{1/q} ds}{  \int_{0}^{\eta^{k-1}}  {( \def\avint{\mathop{\,\rlap{-}\!\!\int}\nolimits} \avint_{Q_s} |f|^q )^{1/q}} ds}\quad \left(\text{by change of variable}\right)
%\notag
%\\ 
%&=\ve_0.
%\notag
%\end{align}
\begin{align}\label{d1}
& \int_{0}^{1}  \left(\aver{Q_s} |f_k|^q\right)^{1/q} ds
\\
& \leq \ve_0  \frac{ \eta^{k-1} \int_{0}^{1}  \left(\aver{Q_s} |f(\eta^{k-1} x, \eta^{2(k-1)}t)|^q dxdt\right)^{1/q} ds}{  \int_{0}^{\eta^{k-1}}  {\left( \aver{Q_s} |f|^q \right)^{1/q}}}\notag
\\
& = \ve_0  \frac{ \int_{0}^{\eta^{k-1}}  (\aver{Q_s} |f|^q )^{1/q} ds}{  \int_{0}^{\eta^{k-1}}  {(\aver{Q_s} |f|^q )^{1/q}} ds}\quad \left(\text{by change of variable}\right)
\notag
\\ 
&=\ve_0.
\notag
\end{align}
Here we have also used  that
% \[
% \omega(\eta^{k-1}) \geq  \frac{1}{\ve_0} \int_{0}^{\eta^{k-1}}  \left( \def\avint{\mathop{\,\rlap{-}\!\!\int}\nolimits} \avint_{Q_s} |f|^q \right)^{1/q}.
% \]
\[
 \omega(\eta^{k-1}) \geq  \frac{1}{\ve_0} \int_{0}^{\eta^{k-1}}  \left(\aver{Q_s} |f|^q \right)^{1/q}.
 \]
Hence, $v$  solves  an equation of the type \eqref{m} such that   the hypothesis in Lemma \ref{rt} is satisfied. Therefore, by applying Lemma \ref{rt}, we obtain that there exists an affine function $Lx= \tilde A x$ such that
\[
\|v- L\|_{L^{\infty}(Q_{\eta})}\leq  \delta_0 \eta^{1+\alpha}.
\]
Scaling back to $u$, we obtain with $L_{k} x\overset{def}= \langle B_k, x\rangle $, where $B_k\overset{def}=\omega(\eta^{k-1})\tilde A$, that 
\begin{equation}\label{est2}
 \|u-L_k\|_{L^{\infty}(Q_{\eta^k})}  \leq \delta_0 \eta^k \eta^{\alpha} \omega( \eta^{k-1}) \leq \delta_0 \eta^k \omega( \eta^k),
 \end{equation}
 where in the last inequality, we used the $\alpha-$decreasing property of $\omega$ (as  in \eqref{dec}). This property is easily seen from the expression of $\omega$  in \eqref{om1} above.  However, since the Statement  $[H]$ does not hold for $i=k$,  we must  necessarily have
 \begin{equation}\label{f0}
 |B_k| \geq 2 \omega(\eta^k).
 \end{equation}
 We now let 
 \[
 \tilde v(x,t)= \frac{u(\eta^k x, \eta^{2k} t)}{\eta^k  \omega(\eta^k)}.
 \]
 Then, we observe that  $\tilde v$ solves
 \[
 \bigg(\delta_{ij}+(p-2) \frac{\tilde v_i \tilde v_j}{|\nabla \tilde v|^2} \bigg)\tilde v_{ij} - \tilde v_t=\frac{\eta^k f(\eta^k x, \eta^{2k} t)}{\omega(\eta^k)}.
 \]
 Moreover, from \eqref{est2} we have, with 
 \begin{equation}\label{a}
 A= \frac{\omega(\eta^{k-1})\tilde A}{\omega(\eta^k)},
 \end{equation}
 that the following inequality holds
 \begin{equation}\label{est4}
 \|\tilde v - \langle A,x\rangle \|_{L^{\infty}(Q_1)} \leq  \delta_0.
 \end{equation}
 Moreover, using that $|\tilde A| \leq C$, where $C$ is universal, and the $\alpha-$decreasing property of $\omega$, we obtain 
 \begin{equation}\label{est5}
 |A| = \frac{|\tilde A| \eta^{\alpha} \omega(\eta^{k-1})}{ \eta^\alpha \omega(\eta^{k})} \leq \frac{C}{\eta^{\alpha}}.
 \end{equation}
 Also \eqref{f0} implies
 \[
 |A| \geq 2.
 \]
 Now again by   change of variables it is seen that   $\tilde f_k$,  defined by
\begin{equation}\label{d10} 
\tilde f_k(x,t)\overset{def}=  \frac{\eta^k f(\eta^k x, \eta^{2k} t)}{ \omega(\eta^k)},
\end{equation}
satisfies the estimate as in \eqref{d1}. 
Now using the fact that $\ve_0 < \delta_0^2$, we find that  $\tilde v$ satisfies the conditions in Lemma \ref{ap2}. Hence,  there exists an affine function $L_0 x\overset{def}=  \langle A_0, x\rangle $, with universal bounds depending on $\eta$, such that
\begin{equation}\label{tv}
|\tilde v(x,t)- L_0(x)| \leq C|X|  K_{\tilde f_k}(|X|),\quad |X|<1,
\end{equation}
where 
%\[
%K_{\tilde f_k}(|X|)= |X|^{\alpha/2} + \int_{0}^{|X|^{1/2}} ( \def\avint{\mathop{\,\rlap{-}\!\!\int}\nolimits} \avint_{Q_s} |\tilde f_k|^q)^{1/q} ds
%\]
\[
K_{\tilde f_k}(|X|)= |X|^{\alpha/2} + \int_{0}^{|X|^{1/2}} \left(\aver{Q_s} |\tilde f_k|^q\right)^{1/q} ds
\]
with $\tilde f_k$ as in \eqref{d10}.  Then, by scaling back to $u$,  we obtain for $|X| \leq \eta^k$ that the following inequality holds by change of variables,
%\begin{align}\label{b1}
%& |u(x,t)- \omega(\eta^k) <A_0, x>| \leq C |X| \left(  \omega(\eta^k) |Y|^{\alpha/2} +  \int_{0}^{\eta^{k} |Y|^{1/2}} ( \def\avint{\mathop{\,\rlap{-}\!\!\int}\nolimits} \avint_{Q_s} |f|^q)^{1/q} ds\right)\  \left(\text{ $Y=(\eta^{-k} x, \eta^{-2k} t)$}\right)
%\\
%& \leq  C |X| \left(  \omega(\eta^{k/2}) |Y|^{\alpha/2} +  \int_{0}^{\eta^{k/2} |Y|^{1/2}} ( \def\avint{\mathop{\,\rlap{-}\!\!\int}\nolimits} \avint_{Q_s} |f|^q)^{1/q} ds\right)\ \left(\text{using $\eta^k \leq \eta^{k/2}$ and $\omega(\eta^{k}) \leq \omega(\eta^{k/2})$}\right)
%\notag
%\\
%& = C |X| \left(  \omega(\eta^{k/2}) |Y|^{\alpha/2} +  \int_{0}^{ |X|^{1/2}} ( \def\avint{\mathop{\,\rlap{-}\!\!\int}\nolimits} \avint_{Q_s} |f|^q)^{1/q} ds\right).
%\notag
%\end{align}
\begin{align}\label{b1}
& |u(x,t)- \omega(\eta^k) \langle A_0, x\rangle| \leq C |X| \left(  \omega(\eta^k) |Y|^{\alpha/2} +  \int_{0}^{\eta^{k} |Y|^{1/2}} \left(\aver{Q_s} |f|^q\right)^{1/q} ds\right)\  \left(\text{ $Y=(\eta^{-k} x, \eta^{-2k} t)$}\right)
\\
& \leq  C |X| \left(  \omega(\eta^{k/2}) |Y|^{\alpha/2} +  \int_{0}^{\eta^{k/2} |Y|^{1/2}} \left(\aver{Q_s} |f|^q\right)^{1/q} ds\right)\ \left(\text{using $\eta^k \leq \eta^{k/2}$ and $\omega(\eta^{k}) \leq \omega(\eta^{k/2})$}\right)
\notag
\\
& = C |X| \left(  \omega(\eta^{k/2}) |Y|^{\alpha/2} +  \int_{0}^{ |X|^{1/2}} \left(\aver{Q_s} |f|^q\right)^{1/q} ds\right).
\notag
\end{align}
Now, let $j$ be the smallest integer such that  $|Y| \leq  \eta^{j}$. Then, we have that
%\begin{align}\label{b8}
%  \omega(\eta^{k/2}) |Y|^{\alpha/2}& \leq \omega(\eta^{k/2}) \eta^{j\alpha/2}
%\\
%& =\frac{1}{\ve_0}  \sum_{i=j/2}^{\frac{k+j}{2}} \eta^{ i\alpha} \omega_1(\eta^{\frac{k+j}{2} - i}) \leq \omega(\eta^{\frac{k+j}{2}})
%\notag
%\\
%& \leq C\bigg[ \bigg(\int_{0}^{ 1} \left( \def\avint{\mathop{\,\rlap{-}\!\!\int}\nolimits} \avint_{Q_s} |f|^q)^{1/q} \bigg) |X|^{\alpha/4} +  \int_{0}^{ |X|^{1/4}} ( \def\avint{\mathop{\,\rlap{-}\!\!\int}\nolimits} \avint_{Q_s} |f|^q\right)^{1/q} ds \bigg] 
%\notag
%\\
%&\leq C K_0(4|X|)\quad (\text{using $Y=(\eta^{-k} x, \eta^{-2k} t)$}),
%\notag
%\end{align}
\begin{align}\label{b8}
  \omega(\eta^{k/2}) |Y|^{\alpha/2}& \leq \omega(\eta^{k/2}) \eta^{j\alpha/2}
\\
& =\frac{1}{\ve_0}  \sum_{i=j/2}^{\frac{k+j}{2}} \eta^{ i\alpha} \omega_1(\eta^{\frac{k+j}{2} - i}) \leq \omega(\eta^{\frac{k+j}{2}})
\notag
\\
& \leq C\bigg[ \left(\int_{0}^{ 1} \left(\aver{Q_s} |f|^q\right)^{1/q} ds\right) |X|^{\alpha/4} +  \int_{0}^{ |X|^{1/4}} \left(\aver{Q_s} |f|^q\right)^{1/q} ds \bigg] 
\notag
\\
&\leq C K_0(4|X|)\quad (\text{using $Y=(\eta^{-k} x, \eta^{-2k} t)$}),
\notag
\end{align}
where the last inequality in \eqref{b8} follows from a computation as in \eqref{b0}. This implies that \eqref{des} holds with $\tilde L x\overset{def}= \langle \omega(\eta^k) A_0, x\rangle $, when $|X| \leq \eta^k$. 

Now when $|X| \geq \eta^k$,   one can show  that 
\begin{equation}\label{cl1}
|u(x,t) | \leq C |X| \omega(2|X|) \leq C |X|  K_0(4 |X|).
\end{equation}
This follows from the fact that with $L_ix\overset{def}= \langle B_i,x\rangle$ we have for $i=0,\ldots, k-1$,
\[
\|u-L_i\|_{L^{\infty}(Q_{\eta^i})}  \leq \delta_0 \eta^i \omega( \eta^i) 
\]
and
\[
| B_i| \leq 2 \omega(\eta^i)
\]
because \eqref{H} holds upto $k-1$.  Moreover, for $i=k$, we  again   have
\[
\|u-L_k\|_{L^{\infty}(Q_{\eta^k})}  \leq \delta_0 \eta^k \omega( \eta^k).
\]
In this case, instead the following bound holds
\[
|B_k| \leq C \omega(\eta^{k-1}) \leq \frac{C \omega(\eta^k)}{\eta^\alpha}\ \text{(using $\alpha-$decreasing property of $\omega$).}
\]
Using these estimates, it is easy to see that   \eqref{cl1}  holds. Now note that with   \linebreak $\tilde L x\overset{def}= \langle \tilde B, x\rangle $, where $\tilde B\overset{def}= \omega(\eta^k) A_0$,  we also  have  the following bound
\begin{equation}\label{ai}
 |\tilde B| \leq C \omega(\eta^k).
 \end{equation}
Therefore, it follows from \eqref{cl1} and the estimate \eqref{ai} above that 
  \begin{equation}
  |u(x,t)- \tilde L(x) | \leq C |X| K_0(4|X|)
  \end{equation}
  also holds when $|X| \geq \eta^k$, for a possibly different $C$.  Hence  the estimate in \eqref{des} follows with $\tilde L x\overset{def}=\langle \tilde B, x\rangle $  and this finishes the proof of the theorem.

\end{proof}

\subsection{Proof of Theorem \ref{main1}}
In this subsection, we assume that $u$ is a $W^{2,1,m}$ viscosity solution to
\begin{equation}\label{m2}
\bigg(\delta_{ij} + (p-2) \frac{u_i u_j}{|\nabla u|^2} \bigg) u_{ij} - u_t =f,
\end{equation}
where $f \in L^{m}$ for some $m>n+2$.  We now state and prove the counterparts of the approximation lemmas in this situation.  The analogue of Lemma \ref{app1} is as follows.

\begin{lem}\label{ap01}
Let $u$ be a  $W^{2,1, m}$ viscosity solution to
\begin{equation}
\bigg(\delta_{ij} + (p-2) \frac{(\delta u_i+A_i)(\delta u_j +A_j)}{ |\delta \nabla u+A|^2} \bigg) u_{ij}- u_t=f\quad \text{in $Q_1$},
\end{equation}
 with $|u| \leq 1$ and $|A| \geq 1$.  Given $\tau>0$, there exists  $\delta_0=\delta_0(\tau)>0$  such that if  
%\[
%\delta, \left(\frac{1}{|Q_{3/4}|} \int_{Q_{3/4}}  |f|^m \right)^{1/m} \leq \delta_0,
%\]
\[
\left(\frac{1}{|Q_{3/4}|} \int_{Q_{3/4}}  |f|^m \right)^{1/m} \leq \delta_0,
\]
then $\|w-u\|_{L^{\infty}(Q_{1/2})} \leq \tau$ for some $w \in C^{2,1}(\overline{Q_{1/2}})$ with universal $C^{2,1}$ bounds depending only on $n, p$ and independent of $|A|$. 
\end{lem}

\begin{proof}
The proof is identical to that of Lemma \ref{app1} and so we omit the details.

\end{proof}

We now state the counterpart of  Lemma \ref{ap2}.

\begin{lem}\label{ap02}
Let $u$ be  a viscosity  solution to
\[
\bigg(\delta_{ij}+ (p-2) \frac{u_i u_j}{|\nabla u|^2} \bigg) u_{ij} - u_t=f
\]
in $Q_1$ with  $u(0,0)=0$. There exists a universal $\delta_0>0$, such that if for some  $A \in \R^n$ satisfying  $M \geq|A| \geq 2$  we have 
\[
\|u- \langle A,x\rangle \|_{L^{\infty}(Q_1)} \leq \delta_0,
\]
and also 
\[
\|f\|_{L^m(Q_1)}  \leq \delta_0^2,
\] 
then there exists an affine function  $ L_0$, with universal bounds depending also on $M$, such that
\begin{equation*}
|u(x,t)-  L_0(x) | \leq C|X|^{1+\alpha_0},
\end{equation*}
where $\alpha_0 < \min(\alpha, 1-\frac{n+2}{m})$. 
\end{lem}

\begin{proof}
As in the proof of Lemma \ref{ap2}, we show that  for every $k \in \mathbb{N}$, there exists affine functions $\tilde L_kx = \langle A_k ,x\rangle $ such that
\begin{equation}\label{itt}
\begin{cases}
\|u- \tilde L_k x\|_{L^{\infty}(Q_{r^k})}\leq \delta_0 r^{k(1+\alpha_0)},
\\
|A_k - A_{k+1}| \leq C \delta_0 r^{k\alpha_0},
\end{cases}
\end{equation}
for some $r<1$ universal independent of $\delta_0$. The conclusion of the lemma then follows from \eqref{itt} in a standard way. 
We first observe that \eqref{itt} holds for $k=0$ with $A_0=A$. Moreover the non-degeneracy condition as in  \eqref{ndeg} is easily verified in this situation provided $\delta_0$ is small enough. Now assume \eqref{itt} holds upto some $k$.  We then define 
\[
v= \frac{u-\tilde L_k (r^k x, r^{2k} t)}{\delta_0r^{k(1+\alpha_0)}}.
\]
Then $v$ solves in $B_1$
\[
\bigg(\delta_{ij} + (p-2) \frac{ (\delta_0 r^{k \alpha_0} v_i + (A_k)_i) (\delta_0 r^{k \alpha_0} v_j + (A_k)_j)}{| \delta_0 r^{k\alpha_0} \nabla v + A_k|^2} \bigg) v_{ij} - v_t= f_k,
\]
where $f_k$ is defined as
\[
f_k(x,t)= r^{k(1-\alpha_0)} \frac{f(r^k x, r^{2k} t)}{\delta_0}.
\]
Now by change of variable it is seen that
\[
\|f_k\|_{L^m(Q_1)} = r^{k(1-\frac{n+2}{m} - \alpha_0)} \frac{1}{\delta_0} \|f\|_{L^m(Q_{r^k})} \leq \delta_0.
\]
Note that over here, we crucially used the hypothesis of the lemma i.e,
\[
\|f\|_{L^m(Q_1)} \leq \delta_0^2,
\]
and the fact that  $\alpha_0 < 1-\frac{n+2}{m}$.  Therefore, $v$ satisfies the hypothesis of Lemma \ref{ap01} and at this point  we can repeat the arguments in the proof of Lemma \ref{ap2} to conclude that there exists $ \tilde L_{k+1}(x) = \tilde L_k (x) + \delta_0 r^{k(1+\alpha_0)}L( \frac{x}{r^k})$, where $L$ has universal bounds such that \eqref{itt} holds for $k+1$. This verifies the induction step and the conclusion of the lemma thus follows.

\end{proof}

We also have the following lemma which is the analogue of Lemma \ref{rt}.

\begin{lem}\label{rt1}
Let $u$ be a solution of
\begin{equation}\label{j10}
\bigg(\delta_{ij}+ (p-2) \frac{u_i u_j}{|\nabla u|^2} \bigg) u_{ij} - u_t=f \quad \text{in $Q_1$},
\end{equation}
 with $|u| \leq 3$ and $u(0,0)=0$.  There exists  a universal $\ve_0>0$ such that if 
\begin{equation}\label{sml5}
\|f\|_{L^m(Q_1)} \leq \ve_0,
\end{equation}
then there exists an affine function $L$  with universal bounds  and a universal $\eta \in (0,1)$  such that
\[
\|u- L\|_{L^{\infty}(Q_{\eta})} \leq \delta_0 \eta^{1+\alpha_0},
\]
where $\delta_0$ is as in Lemma \ref{ap02} above. Without loss of generality we may take $\ve_0 < \delta_0^2$. 

\end{lem}

\begin{proof}
The proof is again identical to that of Lemma \ref{rt} and thus we skip the details. 

\end{proof}

With Lemmas \ref{ap01}--\ref{rt1} in hand, we now proceed with the proof of Theorem \ref{main1}. 

\begin{proof}[ Proof of Theorem \ref{main1}]
It suffices to show that at $(0,0)$, there exists an affine function $\tilde L$  with universal bounds such that 
\begin{equation}\label{des1}
|u(x,t)- \tilde L(x)| \leq C|X|^{1+\alpha_0}. 
\end{equation}
Recall $|X|=\text{max}(|x|, |t|^{1/2})$.  We also assume that $u(0,0)=0$.  Now with  $\eta, \ve_0$ as in Lemma \ref{rt1} and  $\delta_0$ as in  Lemma \ref{ap02},   assume the following hypothesis for a given $i \in \mathbb{N}$,
\begin{equation}\label{h1}[H1]
\begin{cases}
\text{There exists affine function $L_i(x)\overset{def}= \langle B_i, x\rangle $ such that}\  \|u-L_i\|_{L^{\infty}(Q_{\eta^i})}  \leq \frac{\delta_0}{\ve_0} \eta^{i (1+\alpha_0)}
\\
\text{and}\ |B_i|  \leq  \frac{2}{\ve_0} \eta^{i\alpha_0}.
\end{cases}
\end{equation}
By multiplying $u$ with a suitable constant, we may assume that the hypothesis  holds for $i=0$ with $L_0=0$.  We can also assume that 
\begin{equation}\label{nor}
\|f\|_{L^m(Q_1)} \leq 1.
\end{equation}

Let $k$ be the smallest integer such that \eqref{h1} fails. Then as in the proof of Theorem \ref{main}, there are two possibilities.

\medskip

\emph{Case 1:} Suppose $k=\infty$. Then in this case, \eqref{des1} is seen to hold with $\tilde L=0$. 

\medskip

\emph{Case 2:} Suppose instead that $k<\infty$. Then we have that the hypothesis is satisfied upto $k-1$.  As before, we let
\[
v(x,t)= \ve_0 \frac{u(\eta^{k-1}x, \eta^{2(k-1)}t)}{\eta^{(k-1)(1+\alpha_0)}},
\]
which solves in $B_1$
\[
\bigg( \delta_{ij}+ (p-2) \frac{v_i v_j}{|\nabla v|^2} \bigg) v_{ij} - v_t = f_k,
\]
where 
\[
f_k(x,t) = \ve_0 \eta^{(k-1)(1-\alpha_0)} f(\eta^{k-1} x, \eta^{2(k-1)} t).
\]
Then  by change of variable and \eqref{nor}, it  is again seen that  $\|f_k\|_{L^m} \leq \ve_0 $. Moreover, from \eqref{h1} and triangle inequality it follows that $|v| \leq 2 + \delta_0 \leq 3$.  Thus the hypothesis of Lemma \ref{rt1} is satisfied and consequently there exists $L x= \langle \tilde A, x\rangle $ affine such that
\[
\|v- L\|_{L^{\infty}(Q_\eta)} \leq \delta_0 \eta^{1+\alpha_0}.
\]
By scaling back to $u$, we obtain with $L_k x\overset{def}=  B_k x$, with $B_k= \frac{\eta^{(k-1)\alpha_0}}{\ve_0} \tilde A $, that the following holds,
\[
\|u-L_k\|_{L^{\infty}(Q_{\eta^k})} \leq \frac{\delta_0}{\ve_0} \eta^{k(1+\alpha_0)}.
\]
However since Statement $[H1]$ fails,  we  must necessarily have
\[
|B_k| \geq \frac{2}{\ve_0} \eta^{k\alpha_0}.
\]
If we now let
\[
\tilde v(x,t)= \ve_0 \frac{u(\eta^k x, \eta^{2k} t)}{\eta^{k(1+\alpha_0)}},
\]
then, as in the proof of Theorem \ref{main}, it can be easily checked  that $\tilde v$ solves an equation of the type \eqref{m} such that the hypothesis of Lemma \ref{ap02} is verified. Hence there exists an affine function $L_0 x= \langle A_0, x\rangle $,  with universal bounds depending on $\eta$, such that
\[
|\tilde v- L_0 x| \leq C|X|^{1+\alpha_0}.
\]
By scaling back to $u$ we obtain  that, with $\tilde L(x)\overset{def}= \frac{\eta^{k\alpha_0}}{\ve_0}\langle  A_0, x\rangle $, the following estimate holds for $|X| \leq \eta^k$, 
\begin{equation}\label{g10}
|u(x,t)- \tilde L(x)| \leq C|X|^{1+\alpha_0}.
\end{equation}
The rest of the argument is again the same  as in  the proof of Theorem \ref{main}, which allows us to conclude that the estimate \eqref{g10} holds also when $|X| \geq \eta^k$. This finishes the proof of the theorem.

\end{proof}

  \end{document}